\documentclass[reqno]{amsart}
\usepackage[latin1]{inputenc}
 \usepackage[T1]{fontenc}
  \usepackage{amsmath,amsthm,amssymb}
 \usepackage[all]{xy}

\def \N{{\mathbb N}}

\def \R{{\mathbb R}}

\def \1{{\mathbb 1}}

\newtheorem{Exemp}{Examples}
\newtheorem{Thm}{Theorem}
\newtheorem{Prop}{Proposition}
\newtheorem{Def}{Definition}
\newtheorem{Rem}{Remark}
\newtheorem{Lem}{Lemma}

\newtheorem{Def Nota}{Definitions and notations} 
\newtheorem{Cor}{Corollary}

\font\ninerm=cmr9
\long\outer\def\abstract#1{\bigskip\vbox{\noindent\ninerm
\baselineskip=10pt#1}\nobreak\bigskip}
\newcount\numero
\def\exo#1{\advance\numero by 1\bigskip
{\noindent\tenbf #1\the\numero. }}
\def\frac#1#2{{#1\over #2}}
\title{On the Krein-Milman-Ky Fan theorem for convex compact metrizable sets.}   
\author{Mohammed Bachir}
\begin{document}
\maketitle

\begin{center} {\it Laboratoire SAMM 4543, Universit\'e Paris 1 Panth\'eon-Sorbonne, Centre P.M.F. 90 rue Tolbiac 75634 Paris cedex 13}
\end{center}
\begin{center} 
{\it Email : Mohammed.Bachir@univ-paris1.fr}
\end{center}
\noindent\textbf{Abstract.} The Krein-Milman theorem (1940) states that every convex compact subset of a Hausdorff
locally convex topological space, is the closed convex hull of its extreme points. In 1963, Ky Fan extended the Krein-Milman theorem to the general framework of $\Phi$-convexity. Under general conditions on the class of functions $\Phi$, the Krein-Milman-Ky Fan theorem asserts then, that every compact $\Phi$-convex subset of a Hausdorff space, is the $\Phi$-convex hull of its $\Phi$-extremal points. 
We prove in this paper that, in the metrizable case the situation is rather better. Indeed, we can replace the set of $\Phi$-extremal points by the smaller subset of $\Phi$-exposed points. We establish under general conditions on the class of functions $\Phi$, that every $\Phi$-convex compact metrizable subset of a Hausdorff space, is the $\Phi$-convex hull of its $\Phi$-exposed points. As a consequence we obtain that each convex weak compact metrizable (resp. convex weak$^*$ compact metrizable) subset of a Banach space (resp. of a dual Banach space), is the closed convex hull of its exposed points (resp. the weak$^*$ closed convex hull of its weak$^*$ exposed points). This result fails in general for compact $\Phi$-convex subsets that are not metrizable.
\vskip5mm
\noindent{\bf Keyword, phrase:} Extreme points, exposed and $\Phi$-exposed points, $\Phi$-convexity, Krein-Milman theorem, variational principle.\\
{\bf 2010 Mathematics Subject:}  46B22, 46B20, 49J50.

\tableofcontents

\section{Introduction.}

Let $S$ be any nonempty set, $\Phi$ a familly of real valued functions on $S$. A subset $X\subseteq S$ is said to be $\Phi$-convex if $X=S$ or there exists a nonempty set $I$, such that $$X=\cap_{i\in I} \lbrace x\in S: \varphi_i(x)\leq \lambda_i\rbrace$$
where $\varphi_i\in \Phi$ and $\lambda_i\in \R$ for all $i\in I$. For a nonempty set $A\subseteq S$, the intersection of all $\Phi$-convex subset of $S$ contaning $A$ is said to be the $\Phi$-convex hull of $A$. By $\textnormal{conv}_\Phi(A)$ we denote the $\Phi$-convex hull of $A$.

\noindent Let $a, x, y \in S$, we say that $a$ is $\Phi$-between $x$ and $y$, if

$$(\varphi \in \Phi, \varphi(x)\leq\varphi(a), \varphi(y)\leq\varphi(a))\Longrightarrow (\varphi(a)=\varphi(x)=\varphi(y)).$$
Let $\emptyset\neq A \subseteq B\subseteq S$. The set $A$ is said to be $\Phi$-extremal subset of $B$, if 
$$(a\in A, a \textnormal{ is } \Phi\textnormal{-between the points } x,y\in B)\Longrightarrow (x\in A, y\in A).$$
If $A$ is a singleton $A=\lbrace a \rbrace$, we say that $a$ is $\Phi$-extremal point of $B$. The set of all $\Phi$-extremal points of a nonempty set $A$ will be denoted by $\Phi\textnormal{Ext}(A)$.
\vskip5mm
When $S$ is a Hausdorff locally convex topological vector space (in short $\textnormal{l.c.t}$ space, {\it "Hausdorff"} will be implicit), and $\Phi=S^*$ is the topological dual of $S$, then the $\Phi$-convexity and the classical convexity coincides for closed subsets of $S$ and the $\Phi$-convex hull of a set coincides with its closed convex hull. Also, the $\Phi$-extremal points of a nonempty set, coincides with its extreme points [\cite{Kh}, Proposition 2.] ( Note that in \cite{Kh}, the author use a class $\Psi$ of functions in the definitions which correspond to the class $-\Phi$ in the above definitions. The class $\Phi$ concidered in this paper will be a vector space). Recall that if $C$ is a subset of $S$, we say that a point $x\in C$ 
is an extreme point of $C$, and write $x \in \textnormal{Ext}(C)$, if and only if :
$y, z \in C, 0<\alpha <1 ;\hspace{2mm} x = \alpha y+(1-\alpha)z \Longrightarrow x = y = z.$
\vskip5mm

The result in what is known as  the Krein-Milman theorem (1940, \cite{KM}), asserts that if $K$ is a convex compact subset of an $\textnormal{l.c.t}$ space, then $K$ is the closed convex hull of its extreme points,
$$K=\overline{\textnormal{conv}}(\textnormal{Ext}(K)).$$ 
The Krein-Milman theorem has a partial converse known as Milman's theorem (See \cite{P}) which states that if $A$ is a subset of $K$ and the closed convex hull of $A$ is all of $K$, then every extreme point of $K$ belongs to the closure of $A$, 
$$(A\subset K; \hspace{2mm} K=\overline{\textnormal{conv}}(A))\Longrightarrow \textnormal{Ext}(K)\subset \overline{A}.$$ 

In \cite{Ky}, Ky Fan extended the Krein-Milman theorem to the more general framework of $\Phi$-convexity.

\begin{Thm} (Krein-Millman-Ky Fan) Let $S$ be a Hausdorff space and $\Phi$ a familly of real valued functions defined on $S$. Let $K$ be a nonempty compact $\Phi$-convex subset of $S$ and suppose that:

$(1)$ the restriction of each $\varphi \in \Phi$ to $K$, is lower semicontinuous on $K$;

$(2)$ $\Phi$ separate the points of $K$.

\noindent Then,  $\Phi\textnormal{Ext}(K)\neq \emptyset$ and $K=\textnormal{conv}_\Phi(\Phi\textnormal{Ext}(K))$.
\end{Thm}

{\bf The main results of the paper.}
Recall that when $S$ is an $\textnormal{l.c.t}$ space and $C$ is a subset of $S$, we say that a point $x\in C$ is an exposed point of $C$, and write $x \in \textnormal{Exp}(C)$, if there exists some continuous linear functional $x^*\in S^*$ which attains its strict maximum over $C$ at $x$. Such a functional is then said to expose $C$ at $x$. Thus, an exposed point is a special sort of extreme point. If $S$ is a dual space, a weak$^*$ exposed point $x$ (we write $x\in w^*\textnormal{Exp}(C)$) is to simply an exposed point by a continuous functional from the predual. 
We introduce a general concept of $\Phi$-exposed points that coincides with the classical exposed points when $S$ is an $\textnormal{l.c.t}$ space and $\Phi=S^*$ and coincides with the weak$^*$ exposed points where $S=E^*$ is a dual space and $\Phi=E$.

\begin{Def}\label{Def3} Let $S$ be a Hausdorff space, $C$ a nonempty subset of $S$ and $\Phi$ a familly of real valued functions defined on $S$. We say that a point $x$ of $C$ is $\Phi$-exposed in $C$, and write $x\in \Phi\textnormal{Exp}(C)$, if there exists $\varphi\in \Phi$ such that $\varphi$ has a strict maximum on $C$ at $x$ i.e. $\varphi(x)>\varphi(y)$ for all $y\in C\setminus\lbrace x \rbrace$. Such $\varphi$ is then said to $\Phi$-expose $C$ at $x$. 
\end{Def} 
It is easy to see that $\Phi\textnormal{Exp}(C)\subseteq \Phi\textnormal{Ext}(C)$ but the converse is not true in general. The first main result of this paper is the following theorem. 

\begin{Thm} \label{Th1} Let $S$ be a Hausdorff space and $(\Phi,\|.\|_\Phi)$ be a Banach space of real valued functions defined on $S$. Let $K$ be a nonempty compact metrizable $\Phi$-convex subset of $S$ and suppose that:

$(1)$ the restriction of each $\varphi\in \Phi$ to $K$ is continuous, and there exists some real number $\alpha_K\geq 0$ such that $\alpha_K \|\varphi\|_\Phi \geq \sup_{x\in K}|\varphi(x)|$ for all $\varphi\in \Phi$;

$(2)$ $\Phi$ separate the points of $K$.

\noindent Then, we have that

$(i)$  $\Phi\textnormal{Exp}(K)\neq \emptyset$ and the set of all $\varphi\in \Phi$ that $\Phi$-expose $K$ at some point, contains a dense $G_\delta$ subset of $(\Phi,\|.\|_\Phi)$;

$(ii)$ $K=\textnormal{conv}_\Phi(\Phi\textnormal{Exp}(K)).$ 
\end{Thm}
We give below some examples in the linear framework, where the above theorem can be applied (for more details see the corollaries of Subsection \ref{Aff-Ex}).

\begin{Exemp} Let $(E,\|.\|)$ be a Banach space, $S=(E,\textnormal{Weak})$ and $(\Phi,\|.\|_\Phi)= (E^*,\|.\|)$ the topological dual of $(E,\|.\|)$. In this case, a subset of $S$ is $\Phi$-convex iff it is convex and weak closed iff it is convex and norm closed (by Mazur's lemma on the coincidence of weak and norm closure of convex sets). The $\Phi$-convex hull of a set coincides with its weak closed convex hull which also coincides with its norm closed convex hull (by Mazur's lemma). The $\Phi$-exposed points coincides with the classical exposed points.
\end{Exemp}

\begin{Exemp} Let $(E,\|.\|)$ be a Banach space, $S=(E^*,\textnormal{Weak}^*)$ and let $(\Phi,\|.\|_\Phi)=(E,\|.\|)$. In this case, a subset of $S$ is $\Phi$-convex iff it is convex weak$^*$ closed and the $\Phi$-convex hull of a set, coincides with its weak$^*$ closed convex hull. The $\Phi$-exposed points coincides with the weak$^*$ exposed points.
\end{Exemp}

Note that Theorem \ref{Th1} fails in general for compact subsets that are not metrizable. For example, when $(E,\|.\|)=(l^1(\Gamma),\|.\|_1)$ ($\Gamma$ uncountable set), $S=(E^*,\textnormal{Weak}^*)$ and $(\Phi,\|.\|_\Phi)=(E,\|.\|)$, then, all of the hypothesis of Theorem \ref{Th1} are satisfed. However, for the not metrizable weak$^*$ compact subset $K=B_{E^*}$, we have that $\Phi\textnormal{Exp}(K)=w^*\textnormal{Exp}(K)=\emptyset$ (See Remark \ref{R4}).

\vskip5mm
The second main result of this paper is the following theorem. The space $(C(K),\|.\|_{\infty})$ denotes the Banach space of all real valued continuous functions on $K$ equiped with the sup-norm. Let $(\Phi,\|.\|_{\infty})$ be a closed Banach subspace of $(C(K),\|.\|_{\infty})$. By $B_{\Phi^*}$ we denote the dual unit ball of $(\Phi,\|.\|_{\infty})$. We also use the following notation: $$\pm\delta(\Phi\textnormal{Exp}(K)):=\lbrace \pm\delta_k/ k\in \Phi\textnormal{Exp}(K)\rbrace$$ where, for each $k\in \Phi\textnormal{Exp}(K)$, $\delta_k: \varphi \mapsto \varphi(k)$ for all $\varphi\in \Phi$.  
\begin{Thm} \label{Aff-Exp0} Let $K$ be a compact metric space and $(\Phi,\|.\|_{\infty})$ be a closed Banach subspace of $(C(K),\|.\|_{\infty})$ which separates the points of $K$ and contains the constants. Then, we have 
$$w^*\textnormal{Exp}(B_{\Phi^*})=\pm\delta(\Phi\textnormal{Exp}(K)),\hspace{2mm}$$ and 
$$B_{\Phi^*}=\overline{\textnormal{conv}}^{\textnormal{w}^*}(\pm\delta(\Phi\textnormal{Exp}(K))).$$
\end{Thm} 

Note that in Theorem \ref{Aff-Exp0}, the concept of $\Phi$-exposed points of $K$ appears in a natural way in the description of the weak$^*$ exposed points of $B_{\Phi^*}$. As consequence, we deduce, under the hypothesis of Theorem \ref{Aff-Exp0}, that the set $\Phi\textnormal{Exp}(K)$ of all $\Phi$-exposed points of $K$, is a dense subset of the Shilov boundary $\partial \Phi$ of $\Phi$ i.e. $\partial \Phi=\overline{\Phi\textnormal{Exp}(K)}$ (Corollary \ref{Shilov}).  
\vskip5mm 
The results of this paper are based on the following version of variational principle in the compact metric framework (Lemma \ref{MBAD} in Section \ref{Var-Princ}). This analogous of the Deville-Godefroy-Zizler variational principle \cite{DGZ}, also gives a new information about the set of {\it "ill-posed problems"} on compact metric sets. Several others consequences are obtained (the details are given in Section \ref{Var-Princ}).

{\bf A key lemma.}
{\it Let $K$ be a compact metric space and $(Y,\|.\|_Y)$ be a Banach space included in $C(K)$ which separates the points of $K$ and such that $\alpha\|.\|_Y\geq \|.\|_{\infty}$ for some real number $\alpha\geq 0$. Let $f : K\rightarrow \R\cup \left\{+\infty \right\}$ be a proper lower semi-continuous function. Then, the set 
$$N(f)=\left\{\varphi  \in Y : f-\varphi \textnormal { does not have a strict minimum on } K \right\}$$
is of the first Baire category in $Y$. If moreover $Y$ is separable then $N(f)$ can be covered by countably many $d.c.$ hypersurface in $Y$.}
\vskip5mm
In a separable Banach space $Y$, each set $N$ which can be covered by countably many $dc$ $hypersurface$ is $\sigma$-lower porous, also $\sigma$-directionally porous; in particular it is both Aronszajn (equivalent to Gauss) null and $\Gamma$-null. For more details about "small sets", see \cite{Za1} and references therein.
\vskip5mm
This paper is organized as follows. In Section \ref{Var-Princ}, we prove the key lemma (Lemma \ref{MBAD}) and give several consequences.  In Section \ref{The-main} we give the proofs of the main results (Theorem \ref{Th1} and Theorem \ref{Aff-Exp0}) and applications. In Section \ref{Apendix} we give some additional properties and remarks about $\Phi$-exposed points and the Krein-Milman theorem.

\section{Variational principle and consequences.} \label{Var-Princ}
Let $(M,d)$ be a complete metric space and $f : M \longrightarrow \R\cup \left\{+\infty \right\}$ be an
extended real-valued function which is bounded from below and proper. By the term proper we mean that the domain of $f$, $dom(f):=\lbrace x\in M/ f(x) < +\infty\rbrace$ is non-empty. We say that $f$ has a strong minimum at $x$ if $\inf_M f=f(x)$ and $d(x_n,x)\rightarrow 0$ whenever $f(x_n)\rightarrow f(x)$.
The problem to find a strong minimum for $f$, is called {\it Tykhonov well-posed-problem}. Let $(C_b(M), \|.\|_{\infty})$ be the space of all real-valued bounded and continuous functions on $M$, equipped with the sup-norm and let $(Y,\|.\|_Y)$ be a Banach space included in $C_b(M)$. Let 
$$N(f)=\left\{\varphi  \in Y : f-\varphi \textnormal { does not have a strong minimum on }\hspace{1mm} M \right\}.$$
The set $N(f)$ is called the set of {\it "ill-posed problems"}. The problem is to find conditions on $Y$ under which the set $N(f)$ is a {\it "small"} set. In \cite{DGZ}, Deville, Godefroy and Zizler proved that the set $N(f)$ is of first Baire category in $Y$, and in \cite{DR}, Deville and Rivalski generalize the result of Deville-Godefroy-Zizler (D-G-Z), where they showed that the set $N(f)$ is $\sigma$-porous in Y, whenever $f$ is bounded from below, proper and lower semi continuous and $Y$ satisfies the following conditions:

$(i)$ $\|g\|\geq \|g\|_{\infty}$, for all $g\in Y$;

$(ii)$ for every natural number $n$, there exists a positive constant $M_n$ such that for
any point $x \in M$ there exists a function $h_n : M \longrightarrow [0; 1]$, such that $h_n \in Y$ ,
$\|h_n\|\leq M_n$, $h_n(x) = 1$ and $diam(supp (h)) < \frac 1 n.$




\vskip5mm

The D-G-Z variational principle has several applications in particular in optimization and in the geometry of Banach spaces and can be applied without compactness assumption. However, the assumption $(ii)$ in the above result is crutial and so the D-G-Z variational principle cannot includes the linear case, like the Stegall's variational principle. Of course, the interest in the D-G-Z variational principle, is to circumvent the compactness, but in our purpose ( in connection with the Krein-Milman theorem), we need only to treat the compact framework. Thus, we prove in the key Lemma \ref{MBAD} that when we assume that $(K,d)$ is compact metric space, the condition $(ii)$ can be omitted. This allows to expand the class $Y$ to a class of functions including the linear cases. Moreover, if $(Y,\|.\|_Y)$ is a separable Banach space included in $C(K)$ and separate the points of $K$, then the set $N(f)$ can be more smaller than $\sigma$-porous. In fact we prove that in this situation, the set $N(f)$ can be covered by countably many $d.c.$ $hypersurface$ (See the definitions below). This gives, in particular, examples showing that the $\sigma$-porousity of the {\it "ill-posed problems"} in \cite{DR}, is not optimal. Our version of variational principle has several consequences, in particular it allows to give the proofs of the main results of the paper. The proof of Lemma \ref{MBAD}, is based on the use of a differentiability result of convex continuous functions on a separable Banach spaces due to Zajicek \cite{Za} and a non convex analogue to Fenchel duality introduced in \cite{Ba}.  

\vskip5mm

We recall from \cite{Za1} the following definitions.
 \begin{Def} Let $Y$, $Z$ be Banach spaces, $C\subset Y$ an open convex set, and
 $F : C \rightarrow Z$ a continuous mapping. We say that $F$ is d.c. (that is, delta-convex) if there exists a continuous
 convex function $f : C \rightarrow \R$ such that $y^*\circ F+f$ is convex whenever $y^*\in Y^*$, $\|y^*\|\leq 1$.
 \end{Def} 

 \begin{Def} Let $Y$ be a Banach space and $n\in \N^*$, $1 \leq n < dimY$. We say
 that $A \subset X$ is a $d.c.$ $surface$ of codimension $n$ if there exist an $n$-dimensional linear
 space $F \subset X$, its topological complement $E$ and a d.c. mapping
 $\varphi : E \rightarrow F$ such that $A = \left\{x + \varphi(x) : x \in E\right\}$. A d.c. surface of codimension $1$
 will be called a $d.c.$ $hypersurface$.
 \end{Def}

\subsection{Preliminary lemmas.}

 If $(Y,\|.\|_Y)$ is a Banach space included in $C_b(M)$ with $\alpha\|.\|\geq \|.\|_{\infty}$ for some real number $\alpha\geq 0$ and $x\in M$, we denote by $\delta_x$ the evaluation map (Dirac mass) on $Y$ at $x$ i.e. $\delta_x: \varphi \longrightarrow \varphi(x)$, for all $\varphi\in Y$. The map $\delta_x$ is a linear continuous functional on $Y$ since $\alpha\|.\|\geq \|.\|_{\infty}$. We recall the following definition from \cite{Ba}.

\begin{Def} Let $(M,d)$ be a complete metric space and $(Y,\|.\|_Y)$ be a Banach space included in $C_b(M)$ with $\alpha\|.\|\geq \|.\|_{\infty}$ for some real number $\alpha\geq 0$. We say that the space $Y$ has the property $P^G$ if and only if, for every sequence $(x_n)_n\subset M$, the following assertions are equivalent:
\item $(i)$ the sequence $(x_n)_n$ converges in $(M,d)$,
\item $(ii)$ the associated sequence of the Dirac masses $(\delta_{x_n})_n$ converges in $(Y^*, \textnormal{Weak}^*)$.
\end{Def}
The letter $G$ in $P^G$ is justified by the fact that the G\^{a}teaux bornology, the G\^{a}teaux differentiability and the weak$^*$ topology has some connection between them. We refer to \cite{Ba} for more details. The space $C_b(M)$, the subspace $C_b^u(M)$ of uniformly continuous functions and the space $Lip_b(M)$ of all bounded and Lipschitz continuous functions (equipped with their natural norms), satisfies the property $P^G$ for any complete metric space $(M,d)$ (see [Proposition 2.6, \cite{Ba}]). The following Lemma shows that, in the compact metric framework, the property $P^G$ is satisfied for a large class of function spaces.

\begin{Lem} \label{LemmP} Let $(K,d)$ be a compact metric space and $(Y,\|.\|_Y)$ be a Banach space included in $C(K)$, which separates the points of $K$ and such that $\alpha\|.\|\geq \|.\|_{\infty}$ for some real number $\alpha\geq 0$. Then, $Y$ has the property $P^G$.
\end{Lem}
 
\begin{proof}  If $(x_n)_n$ is a sequence of $K$ that converges to some point $x$ in $(K,d)$, it is clear that $(\delta_{x_n})_n$ converge to $\delta_x$ for the weak$^*$ topology. Suppose now that $(\delta_{x_n})_n$ converge to some point $Q$ in $Y^*$ for the weak$^*$ topology. We prove that the sequence $(x_n)_n$ converge in $(K,d)$. Indeed, suppose that $l_1$ and $l_2$ are two distinct cluster point of $(x_n)_n$. There exists two subsequences $(y_n)_n$ and $(z_n)_n$ such that $(y_n)_n$ converge to $l_1$ and $(z_n)_n$ converge to $l_2$ . Since $(\delta_{x_n})_n$ converge to $Q$ and $(Y^*,\textnormal{Weak}^*)$ is a Hausdorff space, it follows that $\delta_{l_1}=Q=\delta_{l_2}$ which is a contradiction since $Y$ separate the points of $K$. So, the sequence $(x_n)_n$ has a unique cluster point, and hence it converges to some point since $K$ is a compact metric space.
\end{proof}

Now, if we are interested on the property $P^G$ for separable Banach spaces $(Y,\|.\|_Y)$ included in $C_b(M)$, the following proposition shows that, this situation holds only when $M$ is compact. In fact, this characterizes the compact metric sets.  

\begin{Prop} \label{PropP1} Let $(K,d)$ be a complete metric space and $(Y,\|.\|_Y)$ be a separable Banach space included in $C(K)$, which separate the points of $K$ and such that $\alpha\|.\|\geq \|.\|_{\infty}$ for some real number $\alpha\geq 0$. Then, the following assertions are equivalent.

$(1)$ $K$ is compact.

$(2)$ $Y$ has the property $P^G$.
\end{Prop}
 
\begin{proof} The part $(1)\Longrightarrow (2)$ is given by Lemma \ref{LemmP}. Let us prove the part $(2)\Longrightarrow (1)$. Indeed, Since $Y$ is separable, by the Banach-Alaoglu theorem, the dual unit ball $B_{Y^*}$ is a compact metrizable space. Let us denotes $\delta(K):=\lbrace \delta_k: k\in K \rbrace$ and consider the map: 
\begin{eqnarray} 
\delta : (K,d) &\rightarrow& (\delta(K),\textnormal{Weak}^*)\nonumber\\
             x &\mapsto& \delta_x\nonumber
\end{eqnarray}
Since $Y$ has the property $P^G$, it follows that $(\delta(K),\textnormal{Weak}^*)$ is a closed subspace of the compact metrizable set $(B_{Y^*},\textnormal{Weak}^*)$. Therefore, $(\delta(K),\textnormal{Weak}^*)$ is a Hausdorff compact space. Since $Y$ separate the points of $K$, the map $\delta$ is one-to-one. Consequently, $ \delta : (K,d) \rightarrow (\delta(X),\textnormal{Weak}^*)$ is a continuous and bijective map from $(K,d)$ onto the compact space $(\delta(K),\textnormal{Weak}^*)$, it is then an homeomorphism which implies that $(K,d)$ is a compact space.

\end{proof}



We also need the following lemma. 
\begin{Lem} \label{Sep} Let $K$ be a compact metric space and $(Y,\|.\|_Y)$ be a Banach space included in $C(K)$. Suppose that $Y$ separate the points of $K$, then there exists a sequence $(\varphi_n)_n\subset Y$ which separates the points of $K$, in particular, there exists a separable Banach subspace $Z$ of $(Y,\|.\|_Y)$ which also separates the points of $K$. 
\end{Lem}
\begin{proof} Since $K$ is a compact metric space, then $(C(K), \|.\|_{\infty})$ is a separable Banach space. Since $Y$ is a subspace of $C(K)$, it is also $\|.\|_{\infty}$-separable. Thus, there exists a sequence $(\varphi_n)_n \subseteq Y$ which is dense in $Y$ for the norm $\|.\|_{\infty}$. Since $Y$ separate the points of $K$, if $x,y \in K$ are such that $x\neq y$, then there exists $\varphi\in Y$ such that $\varphi(x)\neq \varphi(y)$. Using the uniform convergence of a subsequence of $(\varphi_n)_n$ to $\varphi$, we get that the sequence $(\varphi_n)_n$ also separate the points of $K$. Now, if we set $Z=\overline{span}^{\|.\|_Y}\lbrace \varphi_n: n\in \N\rbrace$, then $Z$ is a separable Banach subspace of $(Y,\|.\|_Y)$ that separates the points of $K$.  
\end{proof}

\subsection{The key lemma.}

Recall that in \cite{Za}, Zajicek proved that in a separable Banach space, the set $NG(F)$ of the points where a convex continuous function $F$ is not G\^{a}teaux differentiable, can be covered by countably many $d.c.$ $hypersurface$. This result together with a duality result from cite{Ba}, will be used in the proof of Lemma \ref{MBAD}.

\begin{Lem} \label{MBAD} Let $K$ be a compact metric space and $(Y,\|.\|_Y)$ be a Banach space included in $C(K)$ which separates the points of $K$ and such that $\alpha\|.\|_Y\geq \|.\|_{\infty}$ for some real number $\alpha\geq 0$. Let $f : K\rightarrow \R\cup \left\{+\infty \right\}$ be a proper lower semi-continuous function. Then, the set 
$$N(f)=\left\{\varphi  \in Y : f-\varphi \textnormal { does not have a strong minimum on } K \right\}$$
is of the first Baire category in $Y$. Moreover, for each separable Banach subspace $Z$ of $(Y,\|.\|_Y)$ which separates the points of $K$, we have that $N(f)\cap Z$ can be covered by countably many $d.c.$ hypersurface in $Z$. If particular, if $Y$ is separable then $N(f)$ can be covered by countably many $d.c.$ hypersurface in $Y$.
 \end{Lem}
\begin{proof} Let $Z$ be any Banach subspace of $(Y,\|.\|_Y)$. Consider the function $f^{\times}$ defined for all $\varphi \in Z$ by
 $$ f^{\times} (\varphi):=\sup_{x\in K} \lbrace \varphi(x) - f(x)\rbrace.$$ 
It is clear that $f^{\times}$ is a convex $1$-Lipschitz continuous function on $Z$. 

\noindent{\it \underline{The separable case}:} Suppose that $Z$ is a separable Banach subspace of $(Y,\|.\|_Y)$ which separates the points of $K$. Using [Theorem 2; \cite{Za}] we get that $f^{\times}$ is G\^ateaux-differentiable outside a set $N$ which can be covered by countably many $d.c$ $hypersurface$ in $Z$. On the other hand, combining Lemma \ref{LemmP} and [Theorem 2.8., \cite{Ba}] we get that $f^{\times}$ is G\^ateaux-differentiable at a point $\varphi\in Z$ if and only if $f- \varphi$ has a strong minimum on $K$. Thus, the set $N$ of points of non G\^ateaux-differentiability of $f^{\times}$, coincide with the set 
$$N(f)\cap Z=\left\{\varphi  \in Z : f-\varphi \textnormal { does not have a strong minimum on }\hspace{1mm} K \right\}.$$ This complete the proof in the separable case.

\noindent{\it \underline{The general case}:} As in the proof of Deville-Godefroy-Zizler variational principle, we use the Baire category theorem, but here we dont admits the existence of a bump function in $Y$ having a support of arbitrary small diameter. We use {\it The separable case}. Indeed, we claim that 
$$O_n:=\lbrace \varphi \in Y; \exists x_n\in K/ (f-\varphi)(x_n) < \inf\lbrace (f-\varphi)(x): d(x,x_n)\geq \frac 1 n\rbrace\rbrace$$

is an open dense subset of $Y$. The fact that $O_n$ is open follows from the hypothesis $\alpha\|.\|_Y\geq \|.\|_{\infty}$. Let us prove that $O_n$ is dense. From Lemma \ref{Sep} there exists a separable Banach subspace $Z$ of $(Y,\|.\|_Y)$ which separates the points of $K$. Let $\varphi\in Y$ and $\varepsilon> 0$. By using \noindent{\it "The separable case"} with the function $f-\varphi$ and the separable space $Z$, we get an $h\in Z$ such that $\|h\|_Y<\varepsilon$ and $(f-\varphi)-h$ has a strong minimum on $K$ at some point $x_0$. It follows that $\varphi+h\in O_n$ for all $n\in \N^*$, by taking $x_n=x_0$ for all $n\in \N^*$. This shows that $O_n$ is dense in $Y$. 

Since $Y$ is a complete metric space, $\mathcal{G}=\cap_{n\geq1} O_n$ is a dense $G_\delta$ subset of $Y$, by the Baire category theorem. Now, we claim that 
$$\mathcal{G}\subseteq \left\{\varphi  \in Y : f-\varphi \textnormal { have a strong minimum on } K \right\}.$$ Indeed, let $\varphi \in \mathcal{G} $. For each $n\geq 1$ there exists $x_n \in K$ such that 
$$ (f- \varphi)(x_n) < \inf\lbrace (f-\varphi)(x): d(x,x_n)> \frac 1 n \rbrace.$$
Since $K$ is compact metric space, there exists a subsequence $(x_{n_k})_k$ that converges to some point $x_{\infty}$. Using the lower semicontinuity of $f$, we get that 
\begin{eqnarray*}
(f-\varphi)(x_{\infty}) \leq \liminf_{k} (f-\varphi)(x_{n_k}) &\leq& \liminf_k \inf\lbrace (f-\varphi)(x): d(x,x_{n_k})> \frac 1 n_k\rbrace\\
&\leq& \inf\lbrace (f-\varphi)(x): x\in K\setminus\lbrace x_{\infty}\rbrace\rbrace
\end{eqnarray*}
Now, to see that $x_{\infty}$ is a strong minimum of $f-\varphi$, let $(y_n)_n$ be a sequence in $K$ such that $(f-\varphi)(y_n)$ converges to $(f-\varphi)(x_{\infty})$. We prove that $(y_n)_n$ converges to $x_{\infty}$. Indeed, suppose that the contrary hold. Extracting, if necessary, a subsequence, we can assume that there exists $\varepsilon > 0$ such that for all $n\in N$, $d(y_n,x_{\infty})\geq \varepsilon$. Thus, there exists an integer $p$ such that $d(x_p,y_n)\geq \frac 1 p$ for all $n\in \N$. Hence, 
$$(f-\varphi)(x_{\infty})\leq (f-\varphi)(x_p) < \inf\lbrace (f-\varphi)(x): d(x,x_p)\geq \frac 1 p\rbrace \leq (f-\varphi)(y_n)$$
for all $n\in \N$, which contedict the fact that $\lim_n (f-\varphi)(y_n)=(f-\varphi)(x_{\infty})$. 
This concludes the proof.

\end{proof}

We obtain immediately the following Corollary.

\begin{Cor} \label{P1} Let $(K,d)$ be a compact metric space and $Y$ be any closed subspace of $(C(K), \|.\|_{\infty})$ that separates the points of $K$. Let $f : K\rightarrow \R\cup \left\{+\infty \right\}$ be a proper lower semi-continuous function. Then, the set 
$$N(f)=\left\{\varphi  \in Y : f-\varphi \textnormal { does not have a strong minimum on } K \right\}$$
can be covered by countably many $d.c.$ hypersurface in $Y$.
\end{Cor}
\begin{proof} Since $(K,d)$ is a compact metric space, $(C(K), \|.\|_{\infty})$ is a separable Banach space and so $(Y, \|.\|_{\infty})$ is a separable Banach subspace satisfying the hypothesis of Lemma \ref{MBAD}.

\end{proof}
The above Corollary cannot be obtained from the D-G-Z variational principle. For example, the space $Y=\lbrace\varphi \in C(B_{\R^n})/ \varphi(0)=0\rbrace$, where $B_{\R^n}$ denotes the closed unit ball of $\R^n$ for a fixed norm, satisfies the hypothesis of Corollary \ref{P1} but does not satisfy the condition $(ii)$ of Deville-Rivalski: 

$(ii)$ for every natural number $n$, there exists a positive constant $M_n$ such that for
any point $x \in B_{\R^n}$ there exists a function $h_n : B_{\R^n} \longrightarrow [0; 1]$, such that $h_n \in Y$ ,
$\|h_n\|\leq M_n$, $h_n(x) = 1$ and $diam(supp (h)) < \frac 1 n.$ 
\begin{Rem} A strong and strict minimum coincides for lower semi continuous functions on a compact metric space.
\end{Rem}
\subsection{Applications to the linear variational principle and to the G\^{a}teaux differentiability.}

As consequence of Lemma \ref{MBAD}, we have the following linear analogous to the Stegall's variational principle. Note that also this results cannot be obtained from the D-G-Z variational principle.
\begin{Prop} \label{Boule} Let $E$ be a Banach space and $K$ be a weak$^*$ compact metrizable subset of $E^*$. Let $f: (K,\textnormal{Weak}^*)\longrightarrow \R\cup \left\{+\infty \right\}$ be a proper lower semi-continuous function. Then, the set
$$N(f)=\left\{ x\in E : f-\hat{x} \textnormal { does not have a strict minimum on }\hspace{1mm} K \right\}$$ is included in a set of the first Baire category. If moreover, $E$ is a separable Banach space, then $N(f)$ can be covered by countably many $d.c.$ hypersurface in $E$.
\end{Prop}

\begin{proof} Since $K$ is weak$^*$ compact in $E^*$, it is norm bounded. Let $\alpha:=\sup_{x^*\in K} \|x^*\|< +\infty$. Thus, $(E,\|.\|)$ is a Banach space, included in $C(K)$, which separates the points of $K$ and satisfies $\alpha\|.\|\geq \|.\|_{\infty}$. So we can apply Lemma \ref{MBAD} with $(Y,\|.\|)=(E,\|.\|)$ to obtain the result. 
\end{proof}
\begin{Prop} \label{Boule1} Let $E$ be a Banach space. Let $K$ be a weak compact metrizable subset of $E$. Let $f: (K,\textnormal{Weak})\longrightarrow \R\cup \left\{+\infty \right\}$ be a proper lower semi-continuous function. Then, the set
$$N(f)=\left\{ x^*\in E^* : f-x^* \textnormal { does not have a strict minimum on }\hspace{1mm} K \right\}$$ 
is included in a set of the first Baire category in $E^*$. If moreover, $E^*$ is a separable Banach space, then $N(f)$ can be covered by countably many $d.c.$ hypersurface in $E^*$.
\end{Prop}

\begin{proof} Since $K$ is weak compact in $E^*$, it is norm bounded. Let $\alpha:=\sup_{x\in K} \|x\|< +\infty$. Thus, $(E^*,\|.\|)$ is a Banach space, included in $C(K)$, which separates the points of $K$ by the Hahn-Banach theorem and satisfies $\alpha\|.\|\geq \|.\|_{\infty}$. So we can apply Lemma \ref{MBAD} with $(Y,\|.\|)=(E^*,\|.\|)$ to obtain the result. 
\end{proof}

\begin{Rem} From the above propositions, we get in particular that the set of all continuous functionals from the predual that not exposes the dual unit ball $B_{E^*}$ of a separable Banach space $E$, can be covered by countably many $d.c.$ hypersurface in $E$. Also in a reflexive separable Banach space $E$, the set of all continuous functionals that not exposes the unit ball $B_{E}$ can be covered by countably many $d.c.$ hypersurface in $E^*$.
\end{Rem}
\vskip5mm
Recall that a weak Asplund space $E$ is a Banach space in which every convex continuous function is G\^{a}teaux differentiable at each point of a dense $G_\delta$ subset of $E$. The following corollary, gives a class of convex continuous functions which are G\^{a}teaux differentiable at each point of a dense $G_\delta$ subset of $E$, where $E$ is any Banach space.
\begin{Cor} \label{Diff} Let $E$ be a Banach space and $f : E \longrightarrow \R$, be a convex continuous function such that $\overline{dom(f^*)}^{\textnormal{Weak}^*}$ is weak$^*$ compact metrizable subset of $E^*$. Then, $f$ is G\^{a}teaux differentiable at each point of a dense $G_\delta$ subset of $E$.
\end{Cor}
\begin{proof} Set $K=\overline{dom(f^*)}^{\textnormal{Weak}^*}$, which is a convex weak$^*$ compact metrizable subset of $E^*$. Since, $f^*: K\longrightarrow \R \cup \lbrace+ \infty\rbrace$ is proper weak$^*$ lower semi continuous, by Proposition \ref{Boule} we have that the set
$$G=\lbrace x\in E/ f^* -\hat{x} \textnormal { has a strict minimum on }\hspace{1mm} K\rbrace$$
contain a $G_\delta$ subset of $E$. Since by definition, $dom(f^*)\subset K$, we have that 
$$G=\lbrace x\in E/ f^* -\hat{x} \textnormal { has a strict minimum on }\hspace{1mm} E^*\rbrace.$$
By using the classical Asplund-Rockafellar duality result [Corollary 1, \cite{AR}], we get that $f$ is G\^{a}teaux differentiable at each point of $G$.
\end{proof}

Let $C$ be a non-empty subset of $E^*$. We denote by $\sigma_C$ the support function defined on $E$ by
$$\sigma_C(x)=\sup_{x^*\in C}x^*(x); \hspace{2mm} \forall x\in E.$$

Let $f : E \longrightarrow \R \cup \lbrace+ \infty\rbrace$ be a proper lower semi contiuous convex function. The inf convolution of $f$ and $\sigma_C$ is defined by
$$f\bigtriangledown \sigma_C (x):= \inf_{y\in E} \lbrace f(x-y)+ \sigma_C(y)\rbrace.$$
From the above corollary we get that if $K$ is a convex weak$^*$ compact metrizable subset of $E^*$ and $f$ is  proper lower semi contiuous convex function on $E$, then $f\bigtriangledown \sigma_K$ is G\^{a}teaux differentiable at each point of a dense $G_\delta$ subset of $E$. 

\vskip5mm

Let $K$ be a convex subset of a vector space. A function $\varphi : K\rightarrow \R$ is said to be affine if for all $x, y \in K$ and $0 \leq \lambda \leq 1$, $\varphi(\lambda x + (1 - \lambda)y) = \lambda \varphi(x) + (1 - \lambda)\varphi(y)$. 
The space of all continuous real-valued affine functions on $K$ will be denoted by $Aff(K)$. 
\begin{Prop} \label{Cor0} Let $K$ be a compact metrizable convex subset of a \textnormal{l.c.t} space $X$ and $f: K\longrightarrow \R\cup \left\{+\infty \right\}$ be a proper lower semi-continuous function. Then the set
$$N(f):=\left\{\varphi  \in Aff(K) : f-\varphi \textnormal { does not have a strong minimum on }\hspace{1mm} K \right\}$$ can be covered by countably many $d.c$ hypersurface of  $(Aff(K),\|.\|_{\infty})$. 
\end{Prop}

\begin{proof} We use Lemma \ref{MBAD} with $Y=Aff(K)$. Since $(Aff(K),\|.\|_{\infty})$ is a closed Banach subspace of the separable Banach space $(C(K),\|.\|_{\infty})$, it is separable. On the other hand, by the Hahn-Banach theorem, $Aff(K)$ separate the points of $K$, since it contains the set $\lbrace x^*_{|K}: x^*\in X^*\rbrace$. So, from Lemma \ref{MBAD}, the set 
$$N(f)=\left\{\varphi  \in Aff(K) : f-\varphi \textnormal { does not have a strong minimum on }\hspace{1mm} K \right\}$$ can be covered by countably many $d.c$ hypersurface of  $(Aff(K),\|.\|_{\infty})$.
\end{proof}

\section{The main results and applications.} \label{The-main}
This section is devoted to the proofs of the main results of the paper. Some applications are also given.
\subsection{Proof of Theorem \ref{Th1} and consequences.}\label{Aff-Ex}

Now, we give the proof of the first main result. 

\begin{proof}[Proof of Theorem \ref{Th1}] Let $K$ be a compact metrizable $\Phi$-convex subset of $S$. Under the hypothesis, replacing if necessary $(\Phi,\|.\|_\Phi)$ by the Banach space $(Y,\|.\|_Y)$ where $$Y:=\lbrace \varphi_{|K}, \textnormal{ the restriction of } \varphi\in \Phi \textnormal{ to } K \rbrace$$ with $\|\varphi_{|K}\|_Y:=\|\varphi\|_\Phi$, we can assume that $(\Phi,\|.\|_\Phi)$ is a Banach space included in $C(K)$ which separates the points of $K$ and satisfies $\alpha \|.\|_\Phi\geq \|.\|_{\infty}$. Thus, Lemma \ref{MBAD} applies. Using Lemma \ref{MBAD}, applied with the lower semi continuous function $f=i_K$ (the indicator function which is equal to $0$ on $K$ and $+\infty$ othewize), we get that the set of all $\varphi\in \Phi$ which $\Phi$-exposes $K$ at some point, contains a dense $G_\delta$ subset of $\Phi$. In particular $\Phi\textnormal{Exp}(K)\neq \emptyset$. 

Since $K$ is $\Phi$-convex subset of $S$, it is clear that $\textnormal{conv}_\Phi(\Phi\textnormal{Exp}(K))\subseteq K$. Now, let us prove that $K=\textnormal{conv}_\Phi(\Phi\textnormal{Exp}(K))$. Suppose towards a contradiction that there exists $k_0\in K\setminus \textnormal{conv}_\Phi(\Phi\textnormal{Exp}(K))$. 

{\it Claim.} There exist $h\in \Phi$ and $r\in \R$ such that
\begin{eqnarray*}\label{sep}
\sup \lbrace h(k): k\in \textnormal{conv}_\Phi(\Phi\textnormal{Exp}(K))\rbrace < r < h(k_0).
\end{eqnarray*}

{\it Proof of the claim.} Since $\textnormal{conv}_\Phi(\Phi\textnormal{Exp}(K))$ is a $\Phi$-convex subset of $S$, then there exists a set $I$, $\varphi_i\in \Phi$ and $\lambda_i \in \R$ for all $i\in I$ such that
$$\textnormal{conv}_\Phi(\Phi\textnormal{Exp}(K))=\bigcap_{i\in I}\lbrace k\in S/ \varphi_i(k)\leq \lambda_i\rbrace.$$

Since $k_0\not\in\textnormal{conv}_\Phi(\Phi\textnormal{Exp}(K))$, there exits $i_0\in I$ such that $\varphi_{i_0}(k_0)> \lambda_{i_0}$. On the other hand, we have $\varphi_{i_0}(k)\leq \lambda_{i_0}$ for all $k\in \textnormal{conv}_\Phi(\Phi\textnormal{Exp}(K))$. This finish the proof of the claim by taking $h=\varphi_{i_0}$ and by chosing a real number $r\in \R$ such that $ \lambda_{i_0} < r < \varphi_{i_0}(k_0)$.

Now, using Lemma \ref{MBAD} applied with $f=-h+i_K$, we can find $\psi \in \Phi$ close to $0$ in $\Phi$ such that $h+\psi$,  $\Phi$-expose $K$ at some point $k_1\in \Phi\textnormal{Exp}(K)$. Since $\alpha\|.\|\geq \|.\|_{\infty}$, we have that $\varphi:=h+\psi$ is also close to $h$ uniformly on $K$. Hence, using the claim, $\varphi$ satisfies also
\begin{eqnarray} \label{ex1}
\sup \lbrace \varphi(k): k\in \textnormal{conv}_\Phi(\Phi\textnormal{Exp}(K))\rbrace < r < \varphi(k_0).
\end{eqnarray}
On the other hand 
\begin{eqnarray*}
\varphi(k_0) \leq \sup \lbrace \varphi(k): k\in K\rbrace &=& \varphi(k_1)
\end{eqnarray*}
which is a contradiction with (\ref{ex1}), since $k_1\in \Phi\textnormal{Exp}(K)$.
\end{proof}

For the classical convexity we obtain the following Krein-Milman type results for convex compact metrizable subsets, where the extreme points are replaced by the exposed points. Note that the part $(2)$ in the following corollary is an extension of a Klee result [Theorem 2.1, \cite{K}] (See also [Theorem 4.5, \cite{K}]) and the part $(1)$ sems to be new in a general Banach space. We know from [Theorem 6.2., \cite{Ph}] that a Banach space is a G\^{a}teaux differentiability space if and only if, every convex weak$^*$ compact subset of $E^*$ is the weak$^*$ closed convex hull of its weak$^*$ exposed points. Hence, for a non G\^{a}teaux differentiability space (for example if $E=l^1(\Gamma)$, $\Gamma$ uncountable set) there always exist a convex weak$^*$ compact subset of $E^*$ which is not the weak$^*$ closed convex hull of its weak$^*$ exposed points. This shows in particular that Theorem \ref{Th1} is not true in general for not metrizable convex compact subsets. However, the part $(1)$ of the following corollary, shows that the situation is better for convex weak$^*$ compact metrizable subset of $E^*$, when $E$ is any Banach space.
\begin{Cor} \label{GKlee} Let $E$ be a Banach space.

$(1)$ Let $K$ be a convex weak$^*$ compact metrizable subset of $E^*$. Then,
$$K=\overline{\textnormal{conv}}^{w^*}(w^*\textnormal{Exp}(K)).$$

$(2)$ Let $K$ be a convex weak compact metrizable subset of $E$. Then,

$$K=\overline{\textnormal{conv}}^{w}(\textnormal{Exp}(K))=\overline{\textnormal{conv}}^{\|.\|}(\textnormal{Exp}(K)).$$
\end{Cor}
\begin{proof} In the part $(1)$, we apply Theorem \ref{Th1} with the convex weak$^*$ metrizable subset $K$ of $E^*$ and by taking $(\Phi,\|.\|_\Phi)=(E,\|.\|)$. In the part $(2)$, we apply Theorem \ref{Th1} with the convex weak metrizable subset $K$ of $E$ and by taking $(\Phi,\|.\|_\Phi)=(E^*,\|.\|)$, using in this case the fact that the weak and norm closur coincides for convex sets by the well known Mazur's lemma. 
\end{proof}

In \cite{K}, Klee pointed the fact that outside the normed space, the above result is not true. He suspected that some condition rather close to normability may be needed and that the metrizability is inadequate even in the separable case, mentioning the following counterexample: in the locally convex separable metrizable space $\R^{\aleph_0}$, the cube $[-1,1]^{\aleph_0}$ has no exposed points.  To answer positively this problem in the general \textnormal{l.c.t} spaces, we introduce an intermediate concept of remarkable points called {\it "affine exposed points"} which is between the concept of exposed points and extreme points.  

\begin{Def} \label{Def1} Let $K$ be a convex subset of a \textnormal{l.c.t} space $X$. We say that a point $x\in K$ is an affine exposed point of $K$, and write $x \in \textnormal{AExp}(K)$, if there exists some affine continuous map $\tau \in Aff(K)$ which attains its strict maximum over $K$ at $x$.
\end{Def}
Clearly, $\textnormal{Exp}(K)\subseteq \textnormal{AExp}(K)\subseteq \textnormal{Ext}(K)$, but these inclusions are strict in general. For example, the cube $[-1,1]^{\aleph_0}$ has affine exposed points by Proposition \ref{Cor0}, but is without exposed points. A comparison of these three sets will be given in Subsection \ref{Example}.

We obtain then the following result.
\begin{Cor} \label{K-M-B} Let $K$ be a convex compact metrizable subset of a \textnormal{l.c.t} space $X$. Then, $\textnormal{AExp}(K)\neq \emptyset$ and $K$ is the closed convex hull of its affine exposed points
$$K=\overline{\textnormal{conv}}(\textnormal{AExp}(K)).$$
\end{Cor}
\begin{proof} The proof is given by taking $S=K$ and $(\Phi,\|.\|_\Phi)=(Aff(K),\|.\|_{\infty})$ in Theorem \ref{Th1}.
\end{proof}

 We also have the following consequences.
\begin{Cor} Let $E$ be a Banach space. 

$(1)$ Let $(K, \textnormal{Weak}^*)$ be a convex weak$^*$ compact metrizable subset of $E^*$. Then, the set $w^*\textnormal{Exp}(K)$ is weak$^*$ dense in the set $\textnormal{AExp}(K)$, which is weak$^*$ dense in the set $\textnormal{Ext}(K)$. 

$(2)$ Let $(K, \textnormal{Weak})$ be a convex weak compact metrizable subset of $E$. Then, the set $\textnormal{Exp}(K)$ is  weak dense in the set $\textnormal{AExp}(K)$, which is weak dense in the set $\textnormal{Ext}(K)$. 
\end{Cor} 
\begin{proof} First, note that the spaces $(E^*,\textnormal{Weak}^*)$ and $(E,\textnormal{Weak})$ are \textnormal{l.c.t} spaces. Combining the part $(1)$ (resp. the part $(2)$) of Corollary \ref{GKlee} with Corollary \ref{K-M-B} and the partial converse of the Krein-Milman theorem, we get the part $(1)$ (resp. the part $(2)$).
\end{proof}

\subsection{Proof of Theorem \ref{Aff-Exp0} and the Shilov boundary.}
In this subsection, we give the proof of the second main result. We need the following lemma from \cite{Ba-S}.

\begin{Lem} (See \cite{Ba-S}) \label{maxdiff} Let $Z$ be a Banach space and $ h, k : Z \rightarrow \R $ be two continuous and convex functions. Suppose that the function $z\rightarrow l(z):=max(h(z),k(z))$ is Fr\'echet (respectively, G\^ateaux) differentiable at some point $z_0\in Z$. Then either $h$ or $k$ (maybe both $h$ and $k$) is Fr\'echet (respectively, G\^ateaux) differentiable at $z_0$ and $l'(z_0)=h'(z_0)$ or $l'(z_0)=k'(z_0)$.
 \end{Lem}
\begin{proof}
We give the proof for the Fr\'echet differentiability, the G\^ateaux differentiability is similar. Suppose without loss of generality that  $l(z_0)=h(z_0)$ and let us prove that $h$ is Fr\'echet differentiable at $z_0$ and that $l'(z_0)=h'(z_0)$. For each $z\neq 0$ we have: $$0\leq\frac{h(z_0+z)+h(z_0-z)-2h(z_0)}{\|z\|}\leq \frac{l(z_0+z)+l(z_0-z)-2l(z_0)}{\|z\|}.$$
Since $l$ is convex and Fr\'echet differentiable at $z_0$, then the right-hand side in the above inequalities, tends to $0$ when $z$ tends to $0$. This implies that $h$ is Fr\'echet differentiable at $z_0$ by the convexity of $h$. Now, if we denote $f = h-l$, then
$f(z_0) = 0$, $f\leq 0$ and $f'(z_0)$ exists. Thus, for all $z\in Z$
$$f'(z_0)(z) = \lim_{t\longrightarrow 0^+}\frac{1}{t}(f(z_0 + tz)-f(z_0))\leq 0.$$
This implies that $f'(z_0)=0$. Thus $h'(z_0)=l'(z_0)$.
\end{proof}
We also need to establish the following lemma.
\begin{Lem} \label{Aff1} Let $K$ be a compact metric set and $(\Phi,\|.\|_{\infty})$ be a closed Banach subspace of $(C(K),\|.\|_{\infty})$ which separates the points of $K$ and contains the constants. Then, the following assertions are equivalente.

$(1)$ A point $Q\in B_{\Phi^*}$ is a weak$^*$ exposed point

$(2)$ there exists a $\Phi$-exposed point $k\in \Phi\textnormal{Exp}(K)$ such that $Q=\pm \delta_k$, where $\delta_k: \varphi \mapsto \varphi(k)$ for all $\varphi\in \Phi$.
\end{Lem}
\begin{proof}
$(1)\Longrightarrow (2)$. Let $Q\in w^*\textnormal{Exp}( B_{\Phi^*})$, so there exists $\varphi \in \Phi$ which weak$^*$ expose $ B_{\Phi^*}$ at $Q$. It follows from [Proposition 6.9., \cite{Ph}] that the norm $\|.\|_{\infty}$ is G\^{a}teaux differentiable at $\varphi$ with G\^{a}teaux derivative equal to $Q$. On the other hand it is clear that $\|\psi\|_{\infty}=\max(0^{\times}(\psi),0^{\times}(-\psi))$ for all $\psi\in \Phi$, where $0^{\times}(\psi)=\sup_{k\in K} \varphi (k)$ for all $\varphi\in \Phi$. Thus, from Lemma \ref{maxdiff} we have that either $\psi\mapsto 0^{\times}(\psi)$ or $\psi\mapsto 0^{\times}(-\psi)$ is G\^{a}teaux differentiable at $\varphi$ with G\^{a}teaux derivative equal to $Q$. Suppose in the first case that is the function $\psi\mapsto 0^{\times}(\psi)$ which is G\^{a}teaux differentiable at $\varphi$. Thus, from Lemma \ref{LemmP} and [Theorem 2.8 \cite{Ba}] applied with the function $f=0$, we get that there exists $k\in K$ such that $\varphi$ has a strong maximum at $k$ and that $Q=\delta_k$. Thus, in this case $k$ is $\Phi$-exposed by $\varphi$ and $Q= \delta_k$. For the second case, where it is the function $\psi\mapsto 0^{\times}(-\psi)$ which is G\^{a}teaux differentiable at $\varphi$ with G\^{a}teaux derivative equal to $Q$, in a similar way, using Lemma \ref{LemmP}, [Theorem 2.8 \cite{Ba}] and the chain rule formula we obtain that there exists some $k\in K$ such that $-\varphi$ has a strong maximum at $k$ (so that $k$ is  $\Phi$-exposed point) and $Q=-\delta_k$.

\noindent $(2)\Longrightarrow (1)$. Suppose that $k\in \Phi\textnormal{Exp}(K)$. There exists $\varphi \in \Phi$ which $\Phi$-exposes $k$. Thus $-\varphi$ has a strict minimum at $k$, equivalent to a strong minimum at $k$, since $K$ is compact metric set. We can find a real number $r$ such that $-(\varphi+r)$ has also a strong minimum at $k$ and such that $\varphi+r >1$ on $K$. Hence, the function $0^{\times}$ coincides with $\|.\|_{\infty}$ on an open neighborhood of $\varphi+r\in  \Phi$. Since $-(\varphi+r)$ has a strong minimum at $k$, [Theorem 2.8 \cite{Ba}] asserts that $0^{\times}$ and so also $\|.\|_{\infty}$ is G\^{a}teaux differentiable at $\varphi+r$ with G\^{a}teaux derivative equal to $\delta_k$. It follows from [Proposition 6.9., \cite{Ph}], that $\delta_k$ is weak$^*$ exposed by 
$\varphi+r$. Thus $\delta_k \in w^*\textnormal{Exp}(B_{\Phi^*})$. By the symmetry of $B_{\Phi^*}$, we also have that $-\delta_k \in w^*\textnormal{Exp}( B_{\Phi^*})$.
\end{proof}
{\bf A. The second main result.}
Now, we give the proof of the second main result.
\begin{proof}[Proof of Theorem \ref{Aff-Exp0}]  The first part is given by Lemma \ref{Aff1}. Now, since $(\Phi,\|.\|_{\infty})$ is separable, the weak$^*$ compact set $( B_{\Phi^*},\textnormal{Weak}^*)$ is  metrizable. Thus, from Corollary \ref{GKlee} applied to the convex compact metrizable set $(B_{\Phi^*},\textnormal{Weak}^*)$, we have that 
$$ B_{\Phi^*}= \overline{\textnormal{conv}}^{w^*}(w^*\textnormal{Exp}( B_{\Phi^*})).$$
Thus, we have that
$$ B_{\Phi^*}= \overline{\textnormal{conv}}^{w^*}(w^*\textnormal{Exp}( B_{\Phi^*}))=\overline{\textnormal{conv}}^{w^*}(\pm\delta(\Phi\textnormal{Exp}(K))).$$ This concludes the proof.
\end{proof}

\vskip5mm

We deduce immediately the following corollaries. Replacing $\Phi$ by $Aff(K)$ in Theorem \ref{Aff-Exp0}, we obtain:
\begin{Cor} \label{Aff-Exp} Let $K$ be a compact metrizable convex subset of a \textnormal{l.c.t} space $X$. Then,
$$w^*\textnormal{Exp}(B_{(Aff(K))^*})=\pm\delta(\textnormal{AExp}(K)),\hspace{2mm}$$ and 
$$B_{(Aff(K))^*}=\overline{\textnormal{conv}}^{\textnormal{w}^*}(\pm\delta(\textnormal{AExp}(K))),$$
where $\pm\delta(\textnormal{AExp}(K)):=\lbrace \pm\delta_k/ k\in \textnormal{AExp}(K)\rbrace$.  
\end{Cor} 
\vskip5mm
Replacing $\Phi$ by $C(K)$ in Theorem \ref{Aff-Exp0}, where $(K,d)$ is a compact metric space, and observing that $\Phi\textnormal{Exp}(K)=K$, since each point $k\in K$ is an exposed point by the continuous function $x\mapsto -d(x,k)$, we obtain:
\begin{Cor} \label{CK} Let $(K,d)$ be a compact metric space. Then,
$$w^*\textnormal{Exp}(B_{(C(K))^*})=\pm\delta(K),\hspace{2mm}$$  and
$$B_{(C(K))^*}=\overline{\textnormal{conv}}^{\textnormal{w}^*}(\pm\delta(K)).$$
\end{Cor} 
\vskip5mm
{\bf B. The Shilov boundary and the $\Phi$-exposed points.}
Let $K$ be a compact space and $(\Phi,\|.\|_{\infty})$ be a closed Banach subspace of $(C(K),\|.\|_{\infty})$ which separates the points of $K$. A subset $L$ of $K$ is said to be a norming subset for $\Phi$ if for every $\varphi\in \Phi$, we
have
$$\|\varphi\|_{\infty}=\sup_{x\in L}|\varphi(x)|.$$

A closed subset $C$ of $K$ is a norming subset for $\Phi$ if and only if $C$ is a boundary for $\Phi$, that is, for every $\varphi\in \Phi$, we have
$$\|\varphi\|_{\infty}=\max_{x\in C}|\varphi(x)|.$$
The Choquet boundary of $\Phi$, denoted $Ch(\Phi)$, is defined as
the set of all $x\in K$ such that $\delta_x$ is an extreme point of the unit ball of $\Phi^*$. It is well known that $Ch(\Phi)$ is a boundary for $\Phi$. (See [\cite{T}, p. 184]). When $\Phi$ admits a unique minimal closed boundary, it is called the Shilov boundary of $\Phi$ and is denoted by $\partial \Phi$. D. P. Milman
proved the existence of Silov boundary for every closed linear subspace of
$C(K)$, separating points of $K$ and containing the constants (See \cite{M1} and \cite{M2}). He also proved that in this case the Silov boundary concides with the closure of the Choquet boundary: $\partial \Phi=\overline{Ch(\Phi)}$. A proof of this result, due to H.S. Bear can be found in \cite{Be}. For other informations about boundary sets we refer to \cite{A1} and \cite{A2}. As a consequence of Theorem \ref{Aff-Exp0}, we prove below that if $K$ is a compact metric space, then the set of $\Phi$-exposed points $\Phi\textnormal{Exp}(K)$ is a norming subset for $\Phi$ and that, its closure coincides with the Shilov boundary of $\Phi$. Note that, $\Phi\textnormal{Exp}(K)\subseteq Ch(\Phi)$, but this inclusion is strict in general.
\begin{Cor} \label{Shilov} Let $K$ be a compact metric set and $(\Phi,\|.\|_{\infty})$ be a closed Banach subspace of $(C(K),\|.\|_{\infty})$ which separates the points of $K$ and contains the constants. Then, the set $\Phi\textnormal{Exp}(K)$ is a norming subset for $\Phi$ and we have that $$\partial \Phi= \overline{\Phi\textnormal{Exp}(K)}=\overline{Ch(\Phi)}.$$
\end{Cor}
\begin{proof} By the Hahn-Banach theorem, for each $\varphi\in \Phi$, we have $$\|\varphi\|_{\infty}=\sup_{Q\in B_{\Phi^*}} \langle Q, \varphi\rangle.$$ By Theorem \ref{Aff-Exp0}, we have that $$\|\varphi\|_{\infty}=\sup_{Q\in\overline{\textnormal{conv}}^{\textnormal{w}^*}(\pm\delta(\Phi\textnormal{Exp}(K)))}\langle Q, \varphi\rangle.$$ Since the map $\hat{\varphi}: Q\mapsto \langle Q, \varphi\rangle$ is linear and weak$^*$ continuous, we obtain that $$\|\varphi\|_{\infty}=\sup_{Q\in \pm\delta(\Phi\textnormal{Exp}(K))}\langle Q, \varphi\rangle=\sup_{k\in \Phi\textnormal{Exp}(K)}|\varphi(k)|.$$
Thus, the set $\Phi\textnormal{Exp}(K)$ is a norming subset for $Y$. It follows that $\overline{\Phi\textnormal{Exp}(K)}$ is a boundary for $\Phi$. It is clear that $\partial \Phi\subseteq \overline{\Phi\textnormal{Exp}(K)}$, since $\partial \Phi$ is the minimal closed boundary. We prove that $\Phi\textnormal{Exp}(K)\subseteq \partial \Phi$. Suppose that the contrary hold, there exists $k_0\in \Phi\textnormal{Exp}(K)$ such that $k_0\not\in \partial \Phi$. Thus there exists $\varphi \in \Phi$ that expose $K$ at $k_0$ i.e. $\varphi(k_0)> \varphi(k)$ for all $k\in K\setminus \lbrace k_0\rbrace$. On the other hand, since $\partial \Phi$ is compact, there exists $k_1\in \partial \Phi$ such that $\varphi(k_1)=\sup_{k\in \partial \Phi} \varphi(k)$. Thus $\varphi(k_0)> \varphi(k_1)=\sup_{k\in \partial \Phi} \varphi(k)$. Since $\Phi$ contain the constant, there exists $r\in \R$ such that $\varphi+r\in \Phi$ and $\varphi+r \geq 0$. Hence, we have that 
$$\|\varphi+r\|_{\infty}=\varphi(k_0)+ r> \varphi(k_1) +r=\sup_{k\in \partial \Phi} (\varphi(k) + r)=\|\varphi+r\|_{\infty}$$

which is a contradiction. Thus, $\Phi\textnormal{Exp}(K)\subseteq \partial \Phi$ and so we have $\overline{\Phi\textnormal{Exp}(K)}= \partial \Phi$. 
\end{proof}

\section{Apendix.} \label{Apendix}
In this section, we give some additional properties about remarkable points and the Krein-millman theorem.  
\subsection{Comparison between exposed, affine exposed and extreme points. Examples.} \label{Example}

Let $K$ be a convex subset of a \textnormal{l.c.t} space $X$. It is easy to see that we always have
$$\textnormal{Exp}(K)\subseteq \textnormal{AExp}(K)\subseteq \textnormal{Ext}(K).$$
This section is devoted to give examples showing that these inclusions are strict in general.
\vskip5mm

{\bf A) Example where $\textnormal{Exp}(K)\varsubsetneq \textnormal{AExp}(K)$.}  The cube $[-1,1]^{\aleph_0}$ in the locally convex separable metrizable space $\R^{\aleph_0}$, has no exposed points however the set of its affine exposed points is nonempty. Indeed, for example the point $b=(1,1,1,...)$ is affine exposed in $[-1,1]^{\aleph_0}$ by the affine continuous  map defined on $[-1,1]^{\aleph_0}$ by $\varphi: (x_1,x_2,x_3,...) \mapsto \sum_{n\geq 0} 2^{-n} x_n$.  

A slight change of the set $[-1,1]^{\aleph_0}$, gives also an example where $\emptyset \neq \textnormal{Exp}(K)\neq \textnormal{AExp}(K)$. For example we can take the convex compact set $K:=\lbrace ta+(1-t)k/ t\in [0,1], k\in [-1,1]^{\aleph_0} \rbrace$, where $a=(-2,0,0,0,...)$. In this case the point $a$ is exposed by the continuous functional $x^*: (x_1,x_2,x_3,...) \mapsto -x_1$, but the point $b=(1,1,1,...)$ is not an exposed point. However, $b$ is affine exposed by the affine continuous  map defined on $K$ by $\varphi: (x_1,x_2,x_3,...) \mapsto \sum_{n\geq 0} 2^{-n} x_n$.
\vskip5mm
{\bf B) Example where $\textnormal{Exp}(K)= \textnormal{AExp}(K)$.} Let $K$ be a convex compact subset of an $\textnormal{l.c.t}$ space. Clearly, all translates of continuous linear functionals are
elements of $\textit{Aff}(K)$, but the converse in not true in general (see the above example. See also \cite{P} page 22.). However, we do have the following relationship.

\begin{Prop} \label{Aff} (\cite{P}, Proposition 4.5) Assume that $K$ is a compact convex subset of an \textnormal{l.c.t} space $X$, then
$$L(K):=\left\{ a \in Aff(K) : a = x^*_{|K} + r \hspace{1mm}for \hspace{1mm}some\hspace{1mm} x^* \in X^* \hspace{1mm}and \hspace{1mm}some\hspace{1mm} r \in \R\right\}$$
is dense in $(Aff(K), \|.\|_{\infty})$.
\end{Prop}

If $(E,\|.\|)$ is a Banach space and $E^*$ is its topological dual, the space $X=(E^*,\textnormal{Weak}^*)$ is a $\textnormal{l.c.t}$ space. It is well know that in this case we have that $X^* = E$ (See for instance [Corollary 224., \cite{HHZ}]). In this case, the exposed points of a subset of $X$ coincides, by definition, with the weak$^*$ exposed points and the closure  of a subset coincides with the weak$^*$ closure. 

\begin{Prop} \label{int-nvide} Let $E$ be a Banach space. Let $K$ be a convex weak$^*$ compact subset of $E^*$ such that the norm interior of $K$ is nonempty. Let $X$ be the $\textnormal{l.c.t}$ space $(E^*,\textnormal{Weak}^*)$. Then, $X^*=E$ and $w^*\textnormal{Exp}(K)= \textnormal{AExp}(K)$.

\end{Prop}
\begin{proof}
The fact that $X^*=E$, follows from [Corollary 224., \cite{HHZ}]. Now, let $a$ be a point in the interior of $K$. Replacing $K$ by $K-a$ we can assume without loss of generality that $0$ belongs to the interior of $K$. Thus, from [Corollary 224., \cite{HHZ}], each linear functional that is continuous on $K$, belongs to the space $E$. This shows that 
$$\lbrace x^*_{|K}+r : x^*\in X^*, r\in \R \rbrace=\lbrace \hat{x}_{|K} + r: x\in E, r\in \R \rbrace $$ where $\hat{x}$ denotes the map $x^* \mapsto x^*(x)$ for all $x^*\in E^*$. It is easy to see that the space $\lbrace \hat{x}_{|K} + r: x\in E, r\in \R \rbrace$ equipped with the sup-norm on $K$ is isomorphic to $(E\oplus \R, \|.\|+|.|)$, since $K$ is norm bounded and contain $0$ in its (norm) interior. Hence, $\lbrace x^*_{|K}+r : x^*\in X^*, r\in \R \rbrace$ is a closed Banach subspace of $(Aff(K), \|.\|_{\infty})$. This implies by Proposition \ref{Aff} that $$Aff(K)=\lbrace \hat{x}_{|K} + r: x\in E, r\in \R \rbrace .$$  Note that since $X^*=E$, by definition we have that the weak$^*$ exposed points of the set $K$ considered as a subset of the dual Banach space $(E^*,\|.\|)$ coincides with the exposed points of $K$ considered as a subset of the $\textnormal{l.c.t}$ space $X=(E^*,\textnormal{Weak}^*)$. Note also that a map $\hat{x}_{|K} + r$ affine expose $K$ if and only if $\hat{x}_{|K}$ expose $K$. Hence, $w^*\textnormal{Exp}(K)= \textnormal{AExp}(K)$.

\end{proof}
\begin{Prop} In normed vector space, the exposed points and the affine exposed points coincides for a nonempty convex compact set of finite-dimension (i.e compact convex set whose affine hull is finite-dimensional).
\end{Prop}
\begin{proof} Let $V$ be a normed vector space and $K$ be a convex compact set of finite-dimension. Up to a translation, we can assume without loss of generality that $0\in K$. Let $V_0$ be the linear hull of $K$. Then $K$ has nonempty interior in $V_0$, since $V_0$ is finite-dimentional. Hence, we have that $Aff(K)=\lbrace x^*_{|K}+ r/ x^*\in V_0^*; r\in \R\rbrace$ by Proposition \ref{int-nvide} (since, weak and norm topology coincides in finite-dimentional). Thus, if $\varphi\in Aff(K)$, affine expose $k\in K$, then there exists $x^*\in V_0^*$ and $r\in \R$ such that $x^*_{|K}+ r$ expose $K$ at $k$. This is equivalent to the fact that $x^*$ expose $K$ at $k$ (since $r$ is a constant). Now, by the Hahn-Banach theorem, there exists $\tilde{x}^*\in V^*$ such that $\tilde{x}^*$ coincides with $x^*$ on $V_0$. Hence, $\tilde{x}^*\in V^*$ also expose $K$ at $k$. This shows that $k$ is an exposed point of $K$.
\end{proof}

\begin{Rem} 
We know from Proposition \ref{Aff} that the set $$L(K):=\left\{ a \in Aff(K) : a = r + x^*_{|K} \hspace{1mm}for \hspace{1mm}some\hspace{1mm} x^* \in X^* \hspace{1mm}and \hspace{1mm}some\hspace{1mm} r \in \R\right\}$$ is dense in $(Aff(K), \|.\|_{\infty})$.  As it is given in Proposition \ref{int-nvide}, there exists situations where the sets $L(K)$ and $Aff(K)$ coincides, for instance if $K=B_{E^*}$ in $(E^*,\textnormal{Weak}^*)$, where $E$ is a Banach space. There exist also other situations, where  $L(K)$ can be very "small" subset of $Aff(K)$. Indeed, if $K$ is a compact metrizable subset of a \textnormal{l.c.t} space $X$, without exposed points (for example if $K=[-1,1]^{\aleph_0}$ in $\R^{\aleph_0}$), then from Proposition \ref{Cor0}, we get that $L(K)\subseteq N(0)$ and so $L(K)$ can be covered by countably many $d.c$ hypersurface in $(Aff(K),\|.\|_{\infty})$.
\end{Rem}
{\bf C) Example where $\textnormal{AExp}(K)\varsubsetneq \textnormal{Ext}(K)$.} It is well known that even in the two dimensional space $\R^2$, there exists a closed unit ball $B$ for a suitable norm, such that $\textnormal{Exp}(B)\neq \textnormal{Ext}(B)$ (See for instance Examples 5.9 in \cite{Ph}). Thus by Proposition \ref{int-nvide} we have also that $\textnormal{AExp}(B)\neq \textnormal{Ext}(B)$.

\begin{Rem} \label{R4} Note that the Corollary \ref{GKlee} and Corollary \ref{K-M-B} fails for convex compact sets which are not metrizable. Indeed, take $K=B_{E^*}$ where $E=l^{1}(\Gamma)$ ($\Gamma$ is uncountable set), we know that the norm $\|.\|_1$ is nowhere G\^{a}teaux differentiable (See Example 1.4 (b) p. 3 in \cite{Ph}). So from [Proposition 6.9., \cite{Ph}] we get that the dual unit ball $B_{(l^{1}(\Gamma))^*}$ in the $\textnormal{l.c.t}$ space $((l^{1}(\Gamma))^*, \textnormal{Weak}^*)$, has no ($\textnormal{weak}^*$) exposed points. It follows from Proposition \ref{int-nvide} that $w^*\textnormal{Exp}(B_{(l^{1}(\Gamma))^*})=\textnormal{AExp}(B_{(l^{1}(\Gamma))^*})=\emptyset$. Note also that the assumption of local convexity cannot be omitted. Indeed, Roberts, proved in \cite{R} (1977), that there exist a Hausdorff topological vector space $X$ which is metrizable by a complete metric, and a nonempty compact convex set $K \subseteq X$ such that $\textnormal{Ext}(K)=\emptyset$.
\end{Rem}
\subsection{Remarks on the A.E.P.P spaces.}

We introduce the following class of $\textnormal{l.c.t}$ spaces. 
\begin{Def} An $\textnormal{l.c.t}$ space $X$ is said to have the "\textnormal{Affine Exposed Points Property}" (in short \textnormal{A.E.P.P.}) if and only if every convex compact subset of $X$ is the closed convex hull of its affine exposed points. 
\end{Def}

Let us define  
$$\Xi:=\lbrace X \hspace{2mm} \textnormal{l.c.t}\hspace{2mm} \textnormal{space in which every compact subset is metrizable}\rbrace.$$
 The class  $\Xi$, has been actively studied in the $80'$s years by several authors. This class contains of course all metrizable  \textnormal{l.c.t} spaces, in particular Fr\'echet spaces but is much larger. For several examples, we refer to \cite{CO} and references therein.

\vskip5mm
We obtain immediately from Corollary \ref{K-M-B} the following corollary.
\begin{Cor} Every space from the class $\Xi$, has the A.E.P.P.
\end{Cor}
In particular the space $\R^{\aleph_0}$ has the A.E.P.P. Examples of $\textnormal{l.c.t}$ spaces having the A.E.P.P. who do not belong to the class $\Xi$ are given in Remark \ref{Xi}. For an example of an $\textnormal{l.c.t}$ space without A.E.P.P. we mention the $\textnormal{l.c.t}$ space $((l^{1}(\Gamma))^*, \textnormal{Weak}^*)$, where $\Gamma$ is uncountable (See Remark \ref{Xi} below). Thus, spaces having A.E.P.P. encompasses a broad class of spaces and it would be interesting to better know their properties.
\begin{Exemp} Immediate examples.

$(1)$ Every Fr\'echet space has the A.E.P.P.

$(2)$ Every convex closed and bounded subset of a Fr\'echet-Montel space is the closed convex hull of its affine exposed points (in Fr\'echet-Montel space, any closed bounded set is compact metrizable). A classical example of a Fr\'echet-Montel space is the space $C^{\infty}(\Omega)$ of smooth functions on an open set $\Omega$ in $\R^n$.
\end{Exemp}

Recall that a Banach space $E$ is said to be a G\^{a}teaux differentiability space (GDS) iff each convex continuous real valued function defined on $E$ is G\^{a}teaux differentiable at each point of a dense subset. In \cite{Ph}, Phelps proved the following result.

\begin{Thm} \label{Ph} ([Theorem 6.2., \cite{Ph} p. 95]) A Banach space $E$ is a GDS if and only if every weak$^*$ compact
convex subset of $E^*$ is the weak$^*$ closed convex hull of its weak$^*$ exposed points.
\end{Thm}

\begin{Rem} \label{Xi} $(1)$ Since the exposed points are in particular affine exposed points, it follows from the above theorem that the space $(E^*,\textnormal{Weak}^*)$ has the A.E.P.P. whenever $E$ is a GDS. However, if $E$ is a non separable GDS, the dual unit ball is a $\textnormal{weak}^*$ compact not metrizable subset. Thus, the space $(E^*,\textnormal{Weak}^*)$ has the A.E.P.P. but $(E^*,\textnormal{Weak}^*)\notin \Xi$, whenever $E$ is a nonseparable GDS (For example the nonseparable Hilbert spaces). 

$(2)$ The $\textnormal{l.c.t}$ space $((l^{1}(\Gamma))^*, \textnormal{Weak}^*)$ does not have the A.E.P.P. (See Remark \ref{R4}). More generally, the $\textnormal{l.c.t}$ space $(E^*, \textnormal{Weak}^*)$ does not have the A.E.P.P. whenever $E$ is a Banach space equipped with a nowhere G\^{a}teaux differentiable norm.
\end{Rem}


For examples of not metrizable spaces which belongs to the class $\Xi$, we have for example:
\begin{Prop} \label{fin} Let $E$ be a separable Banach space. Then $(E^*,\textnormal{Weak}^*)$ and $(E,\textnormal{Weak})$ belongs to the class $\Xi$, in particular, they have the A.E.P.P. but are not metrizable.
\end{Prop}
\begin{proof} It is well known that the whole spaces $(E^*,\textnormal{Weak}^*)$ and $(E,\textnormal{Weak})$ are not metrizable. It is also well known that a Banach space $E$ is separable iff every compact subset of $(E^*,\textnormal{Weak}^*)$ is metrizable. Thus, $(E^*,\textnormal{Weak}^*)\in \Xi$. For the space $(E,\textnormal{Weak})$, let $K$ be a weak compact subset of $E$. Since $E$ is separable, then $K$ is also separable. Now, consider $K$ as a subset of $E^{**}$ by the canonical embedding, we get that $K$ is norm separable and weak$^*$ compact subset of $E^{**}$, which implies from [Lemma 2, \cite{Ba1}] that $K$ is weak$^*$ metrizable in $E^{**}$. In other words, $K$ is weak metrizable. Thus $(E,\textnormal{Weak})\in \Xi$.

\end{proof}

Several others not trivial examples of spaces from $\Xi$ can be found in \cite{CO}.

\bibliographystyle{amsplain}

\end{document}